\documentclass[12pt]{article}

\title{On the roots of the subtree polynomial}

\author{Jason I.\ Brown\footnote{Supported by an NSERC grant CANADA, Grant number RGPIN 170450-2013}\\
\small Dalhousie University\\
\small jason.brown@dal.ca\\[10pt]
Lucas Mol\\
\small University of Winnipeg\\
\small l.mol@uwinnipeg.ca}
\date{}

\usepackage[margin=1.17in]{geometry}

\usepackage{amsmath}
\usepackage{amsthm}
\usepackage{amsfonts}

\usepackage[numbers,sort&compress]{natbib}

\newtheorem{theorem}{Theorem}[section]
\newtheorem{lemma}[theorem]{Lemma}

\newtheorem{corollary}[theorem]{Corollary}
\newtheorem{proposition}[theorem]{Proposition}
\newtheorem{conjecture}[theorem]{Conjecture}

\theoremstyle{definition}

\usepackage{color}

\usepackage{tikz}
\tikzstyle{vertex}=[circle, draw, inner sep=0pt, minimum size=4pt,fill=black]
\newcommand{\vertex}{\node[vertex]}
\usetikzlibrary{decorations.pathreplacing}

\usepackage{enumitem}

\begin{document}

\maketitle

\begin{abstract}
For a tree $T$, the \emph{subtree polynomial} of $T$ is the generating polynomial for the number of subtrees of $T$. We show that the complex roots of the subtree polynomial are contained in the disk $\left\{z\in\mathbb{C}\colon\ |z|\leq 1+\sqrt[3]{3}\right\}$, and that $K_{1,3}$ is the only tree whose subtree polynomial has a root on the boundary.  We also prove that the closure of the collection of all real roots of subtree polynomials contains the interval $[-2,-1]$, while the intervals $(\infty,-1-\sqrt[3]{3})$, $[-1,0)$, and $(0,\infty)$ are root-free.
\end{abstract}

\section{Introduction}

For a given (finite) tree $T$, let $s_k(T)$ denote the number of subtrees (i.e.\ connected induced subgraphs) of $T$ of order $k$.  The \emph{subtree polynomial} of $T$, denoted $\Phi_T(x)$, is the generating polynomial for the number of subtrees of $T$, that is,
\[
\Phi_T(x)= \sum_{k=1}^n s_k(T)x^k.
\]
For example, one can easily verify that for the star $K_{1,n-1}$ and the path $P_n$ of order $n$, we have
\[
\Phi_{K_{1,n-1}}(x)=x(1+x)^{n-1}+(n-1)x,
\]
and 
\[
\Phi_{P_n}(x)=nx+(n-1)x^{2} + \dots + 2x^{n-1}+x^{n}=\sum_{k=1}^n(n-k+1)x^k.
\]
A linear time algorithm for computing the subtree polynomial was given in~\cite{YanYeh2006}. 

The coefficients of the subtree polynomial were the subject of intense study by Jamison~\cite{Jamison1983,Jamison1984,Jamison1987,Jamison1990}.  Recently, Ralaivaosaona and Wagner studied the distribution of the subtree orders of a tree~\cite{RalaivaosaonaWagner2018}.  Roughly speaking, they demonstrated that if a tree has sufficiently many leaves and no long branchless paths, then the distribution of the subtree orders is close to a Gaussian distribution.  This relates to a question of Jamison regarding when the coefficients $s_2(T),s_3(T),\dots, s_n(T)$ form a unimodal sequence (c.f.~\cite{RalaivaosaonaWagner2018}).

The subtree polynomial of $T$ encodes several useful invariants of $T$, including the total number of subtrees of $T$, the mean subtree order of $T$, and the independence number of $T$.  The total number of subtrees of $T$, given by $\Phi_T(1),$ has become an important topological index.  Much work has been done on finding the tree(s) in a given family that maximize or minimize the number of subtrees~\cite{SillsWang2015,YanYeh2006,ZhangZhang2015,ZhangZhangGrayWang2013}.  The \emph{mean subtree order} of $T$, introduced by Jamison~\cite{Jamison1983} as a rough measure of the shape of the lattice of subtrees of $T$ is given by $\frac{\Phi'_T(1)}{\Phi_T(1)}$.  Jamison posed six open problems on the mean subtree order in~\cite{Jamison1983}, five of which have recently been solved, along with some subsequent questions~\cite{MolOellermann2017,VinceWang2010,WagnerWang2014,WagnerWang2016,Haslegrave2014}.  Finally, it was shown in~\cite{Jamison1987} that $\Phi_T(-1)=-\alpha(T)$, where $\alpha(T)$ is the independence number of $T$.  Given the interest in various evaluations of the subtree polynomial, it is natural to ask about other properties of the subtree polynomial.

For many graph polynomials, such as chromatic~\cite{DongKohTeo2005}, reliability~\cite{BrownMol2017}, domination~\cite{BrownTufts2014}, 
edge cover~\cite{CsikvariOboudi2011} and independence polynomials~\cite{LevitMandrescu2005}, there has been significant interest in the location and distribution of their roots.  In this article, we study the location and distribution of the roots of the subtree polynomial.  For a tree $T$, we refer to the roots of the subtree polynomial of $T$ as the \emph{subtree roots} of $T$.

We note that the subtrees of a tree $T$ are exactly the connected induced subgraphs of $T$, so the subtree polynomial of $T$ is precisely the \emph{connected set polynomial} (the generating polynomial for the number of connected induced subgraphs) of $T$.  The roots of the connected set polynomial were studied in~\cite{BrownMol2016}, in connection with the node reliability polynomial.  It was shown that the connected set polynomial of every graph of order at least $3$ has a nonreal root (hence every tree of order at least $3$ has a nonreal subtree root).  It was also shown that the closure of the collection of roots of connected set polynomials of all connected graphs is dense in the entire complex plane~\cite{BrownMol2016}.  However, we demonstrate here that the roots of the subtree polynomial are bounded in modulus by the constant $1+\sqrt[3]{3}$.  See Figure~\ref{Roots14} for an illustration of the subtree roots of all trees of order at most $14$ in the complex plane.

\begin{figure}[htb]
    \centering
    \begin{tikzpicture}
    \node[inner sep=0pt,left] (plot) at (0,0)
    {\includegraphics[scale=0.5]{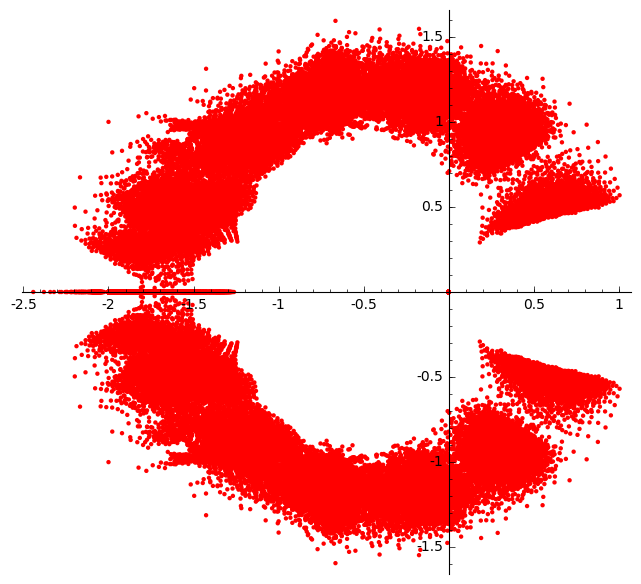}};
    \node[right] at (-0.05,-0.05) {\footnotesize $\mbox{Re}(z)$};
    \node[above] at (-2.4,3.55) {\footnotesize $\mbox{Im}(z)$};
    \end{tikzpicture}
    \caption{The subtree roots of all trees of order at most $14$.}
    \label{Roots14}
\end{figure}


The layout of the article is as follows.  In Section~\ref{BoundSection}, we prove that for any tree $T$, the subtree roots of $T$ lie in the disk $\left\{z\in\mathbb{C}\colon |z|\leq 1+\sqrt[3]{3}\right\}$, and that $K_{1,3}$ is the only tree with a subtree root on the boundary of the disk.  
In Section~\ref{RealLineSection}, we prove that the intervals $(-\infty,-1-\sqrt[3]{3})$, $[-1,0)$ and $(0,\infty)$ are free of subtree roots, and that the closure of the collection of all real subtree roots contains the interval $[-2,-1].$  In Section~\ref{Conclusion}, we discuss several open problems and conjectures.

We require several short definitions and observations which have all appeared in the literature (see~\cite{Jamison1983}, for example).  If $v$ is a vertex of $T$, then for every $k\in\{1,\dots,n\}$, let $s_k(T,v)$ denote the number of subtrees of $T$ of order $k$ that contain $v$.  Then the \emph{local subtree polynomial} of $T$ at $v$ is defined by
\[
\Phi_{T,v}(x)=\sum_{k=1}^n s_k(T,v)x^k.
\]
We also extend the definition of subtree polynomial to forests, mainly for convenience, as forests arise as vertex-deleted subgraphs of trees.  If $F$ is a forest with component trees $T_1,T_2,\dots, T_k$, then the subtree polynomial of $F$ is given by
\[
\Phi_{F}(x)=\sum_{i=1}^k \Phi_{T_i}(x).
\]

Let $T$ be a tree (or forest) with vertex $v$.  Then it is easy to see that 
\begin{align}
\Phi_{T}(x)& = \Phi_{T,v}(x)+\Phi_{T-v}(x), \label{PhiDecomposition}
\end{align}
since every subtree of $T$ either contains $v$ or does not contain $v$.  Moreover, if $v$ has degree $d$ in $T$, and neighbours $u_1,\dots,u_d$ belonging to components $T_1,\dots,T_d$ of $T-v$, respectively, then
\begin{align}
\Phi_{T,v}(z) & = z\prod_{i=1}^d\left[1+\Phi_{T_i,u_i}(z)\right]. \label{PhiProduct}
\end{align}

\section{Subtree roots are bounded in modulus}\label{BoundSection}

In this section, we prove that for any tree $T$, the subtree roots of $T$ lie in the disk $\left\{z\in\mathbb{C}\colon\ |z|\leq 1+\sqrt[3]{3}\right\}$, and that $K_{1,3}$ is the only tree with a subtree root on the boundary of the disk.  The proof technique is somewhat similar to that of Csikvari and Oboudi's proof that the roots of the edge cover polynomial are bounded in modulus~\cite{CsikvariOboudi2011}.  We begin with a lower bound on the modulus of the local subtree polynomial $\Phi_{T,v}(z)$ for $|z|\geq 2$, which improves a bound from~\cite{MolThesis}.

\begin{lemma}\label{LocalBound}
Let $T$ be a tree of order $n$, and let $z\in\mathbb{C}$ with $|z|\geq 2$.  Then for every vertex $v$ of $T$, $|\Phi_{T,v}(z)|\geq |z|\cdot (|z|-1)^{n-1}$.
\end{lemma}

\begin{proof}
Throughout the proof, let $|z|\geq 2$.  The proof is by induction on $n$.  If $n=1$, then
\[
\left|\Phi_{T,v}(z)\right|=|z| = |z|(|z|-1)^{1-1},
\]
so the statement is verified for $n=1$.
Now for some $n\geq 2$, suppose that the statement holds for all trees of order strictly less than $n$.  Let $T$ be a tree of order $n$ with vertex $v$.  Let the neighbours of $v$ be $u_1,\hdots,u_d.$  Recall from (\ref{PhiProduct}) that
\begin{align*}
\Phi_{T,v}(z)=z\cdot\prod_{i=1}^{d}\left[1+\Phi_{T_i,u_i}(z)\right],
\end{align*}
where $T_i$ is the component of $T-v$ containing $u_i.$  Thus we have
\begin{align*}
\left|\Phi_{T,v}(z)\right|&=|z|\cdot\prod_{i=1}^d\left|1+\Phi_{T_i,u_i}(z)\right|\geq |z|\cdot\prod_{i=1}^d\left[ \left|\Phi_{T_i,u_i}(z)\right|-1\right] .
\end{align*}
Let $n_i$ denote the order of tree $T_i$.  Then by the induction hypothesis,
\begin{align*}
\left|\Phi_{T,v}(z)\right| &\geq |z|\cdot \prod_{i=1}^d\left[|z|\cdot (|z|-1)^{n_i-1}-1\right].
\end{align*}
Since $|z|\geq 2$, it follows that
\begin{align*}
\left|\Phi_{T,v}(z)\right|&\geq |z|\cdot \prod_{i=1}^d\left[|z|\cdot (|z|-1)^{n_i-1}-(|z|-1)^{n_i-1}\right]\\
&\geq |z|\cdot \prod_{i=1}^d(|z|-1)^{n_i}\\
&=|z|\cdot (|z|-1)^{n-1}.\qedhere
\end{align*}
\end{proof}

The following corollary is immediate.

\begin{corollary}
Let $T$ be a tree of order $n$ with vertex $v$.  Then the roots of the local subtree polynomial of $T$ at $v$ lie in the disk $\{z\in\mathbb{C}\colon\ |z|<2\}$. \hfill \qed
\end{corollary}

To prove the main result of this section, we demonstrate that for every tree $T$ and every vertex $v$ of $T$, we have $\left|\Phi_{T,v}(z)\right|>\left|\Phi_{T-v}(z)\right|$ for all $|z|>1+\sqrt[3]{3}$, from which it follows that $\left|\Phi_{T}(z)\right|>0$, and thus $\Phi_T(z)\neq 0$.  The argument is by induction on the order of $T$.  We actually consider the possibility that $|z|=1+\sqrt[3]{3}$ as well, so that we can prove that $K_{1,3}$ is the only tree with a subtree root of this modulus.  

Before we proceed with the proof, we handle some base cases.  First, we deal with the tree $K_{1,3}$.  It is easily verified that $\Phi_{K_{1,3}}(z)=z^4+3z^3+3z^2+4z$ has roots $0$, $-1-\sqrt[3]{3}$, and $-1+\left(\sqrt[3]{3}\pm i\sqrt[6]{3^5}\right)/2$, so the following lemma is best possible.

\begin{lemma}\label{K13Lemma}
Let $v$ be a vertex of $T=K_{1,3}$.  If $|z|\geq 1+\sqrt[3]{3},$ then $\left|\Phi_{T,v}(z)\right|\geq \left|\Phi_{T-v}(z)\right|$.
\end{lemma}

\begin{proof}
Let $|z|\geq 1+\sqrt[3]{3}$.  First suppose that $v$ is the vertex of degree $3$ in $T$.  Then 
\[
\left|\Phi_{T,v}(z)\right|=|z|\cdot |z+1|^{3}\geq |z|\cdot \left(|z|-1\right)^{3}\geq 3|z|=\left|\Phi_{T-v}(z)\right|.
\]

Now suppose that $v$ is a leaf of $T$.  For ease of notation, let $c=\sqrt[3]{3}$.  Note that $\Phi_{T,v}(z)=z^4+2z^3+z^2+z$ and $\Phi_{T-v}(z)=z^3+2z^2+3z$, so we show
\begin{align}\label{K13}
|z^3+2z^2+z+1|^2\geq |z^2+2z+3|^2.
\end{align}
Write $z=x+iy$, where $x,y\in\mathbb{R}$.  By substituting and simplifying, we find that (\ref{K13}) is equivalent to 
\begin{align}\label{K13sub}
\begin{split}
&(x^3+3x^2+3x+4)(x^3+x^2-x-2)\\
&\ \ +y^2\left[3x^4+8x^3+6x^2-6x-1+y^2\left(3x^2+4x+1+y^2\right)\right]\geq 0.
\end{split}
\end{align}
First we show that the expression in the square brackets of (\ref{K13sub}) is positive.  Since $|z|\geq  (1+c)$, we have $y^2\geq (1+c)^2-x^2$. Using this inequality repeatedly, we find
\begin{align*}
&\mathrel{\phantom{=}}3x^4+8x^3+6x^2-6x-1+y^2\left(3x^2+4x+1+y^2\right)\\
&\geq 3x^4+8x^3+6x^2-6x-1+y^2\left(2(x+1)^2+c^2+2c\right)\\
&\geq 3x^4+8x^3+6x^2-6x-1+\left((1+c)^2-x^2\right)\left(2(x+1)^2+c^2+2c\right)\\
&=x^4+4x^3+\left(c^2+2c+6\right)x^2+2(2c^2+4c-1)x+(c^4+4c^3+7c^2+6c+1)\\
&=x^4+4x^3+\left(c^2+2c+6\right)x^2+2(2c^2+4c-1)x+(7c^2+9c+13),
\end{align*}
where the fact that $c^3=3$ was used at the last step.  We verify that this last expression is positive for all $x\in\mathbb{R}$. Now we show that the left hand side of (\ref{K13sub}) is nonnegative.  If $x\geq -1-c$, then using $y^2\geq (1+c)^2-x^2$, we have
\begin{align*}
&\mathrel{\phantom{=}}(x^3+3x^2+3x+4)(x^3+x^2-x-2)\\
&\ \ \ \ +y^2\left[3x^4+8x^3+6x^2-6x-1+y^2\left(3x^2+4x+1+y^2\right)\right]\\
&\geq (x^3+3x^2+3x+4)(x^3+x^2-x-2)\\
&\ \ \ \ +\left((1+c)^2-x^2)\right)\left[x^4+4x^3+\left(c^2+2c+6\right)x^2+2(2c^2+4c-1)x+(7c^2+9c+13)\right]\\
&=2(x+1+c)\left(4x^2 + 2(c^2 - 2)x + 7c^2 + 12c + 16\right),
\end{align*}
which we verify is nonnegative for $x\geq -1-c$. On the other hand, if $x<-1-c$, then $y^2\geq 0$, hence the left hand side of (\ref{K13sub}) is at least $(x^3+3x^2+3x+4)(x^3+x^2-x-2)$, which we verify is positive for $x<-1-c$.
\end{proof}

\begin{lemma}\label{BaseCase}
Let $T$ be a tree of order at most $4$ not isomorphic to $K_{1,3}$, and let $v$ be a vertex of $T$.  If $|z|\geq 1+\sqrt[3]{3}$, then $\left|\Phi_{T,v}(z)\right|> \left|\Phi_{T-v}(z)\right|$. 
\end{lemma}

\begin{proof}
Let $|z|\geq 1+\sqrt[3]{3}$.  If $T$ has order $1$, then $\left|\Phi_{T,v}(v)\right|=|z|>0=\left|\Phi_{T-v}(z)\right|$.  If $T\cong K_{1,n-1}$ for some $n\in\{2,3\}$, and $v$ is a vertex of degree $n-1$ in $T$, then 
\[
\left|\Phi_{T,v}(z)\right|=|z|\cdot |z+1|^{n-1}\geq |z|\cdot \left(|z|-1\right)^{n-1}\geq|z|\cdot \left(\sqrt[3]{3}\right)^{n-1}>|z|\cdot(n-1)=\left|\Phi_{T-v}(z)\right|.
\]
There are three remaining cases.

\medskip

\noindent
\textbf{Case I:} $T\cong K_{1,2}$ and $v$ is a leaf.  We prove the stronger statement that if $|z|\geq 2$, then $\left|\Phi_{T,v}(z)\right|> \left|\Phi_{T-v}(z)\right|$.  So let $|z|\geq 2$.  Note that $\Phi_{T,v}(z)=z^3+z^2+z$ and $\Phi_{T-v}(z)=z^2+2z$, so we need to show that 
\begin{align}\label{P3}
|z^2+z+1|^2> |z+2|^2.
\end{align}
Write $z=x+iy$, where $x,y\in\mathbb{R}$.  By substituting and simplifying, we find that (\ref{P3}) is equivalent to 
\[
(x^2-1)(x^2+2x+3)+y^2(2x^2+2x-2+y^2)> 0.
\]
Since $|z|\geq 2$, we have $y^2\geq 4-x^2$.  Using this inequality along with the fact that $x^2+2x+2=(x+1)^2+1> 0$ for all $x\in\mathbb{R}$, we find
\begin{align*}
&\mathrel{\phantom{=}} (x^2-1)(x^2+2x+3)+y^2(2x^2+2x-2+y^2)\\
&\geq (x^2-1)(x^2+2x+3)+y^2(x^2+2x+2)\\
&\geq (x^2-1)(x^2+2x+3)+(4-x^2)(x^2+2x+2)\\
&= 4x^2 + 6x + 5,
\end{align*}
which is easily verified to be strictly positive for all $x\in\mathbb{R}$.

\medskip

\noindent
\textbf{Case II:} $T\cong P_4$ and $v$ is a leaf. We prove the stronger statement that if $|z|\geq 2$, then $\left|\Phi_{T,v}(z)\right|> \left|\Phi_{T-v}(z)\right|$.  So let $|z|\geq 2$. Note that $\Phi_{T,v}(z)=z^4+z^3+z^2+z$ and $\Phi_{T-v}(z)=z^3+2z^2+3z$, so we need to show that 
\begin{align}\label{P4}
|z^3+z^2+z+1|^2> |z^2+2z+3|^2.
\end{align} 
Write $z=x+iy$, where $x,y\in\mathbb{R}$.  By substituting and simplifying, we find that (\ref{P4}) is equivalent to 
\begin{align}\label{P4sub}
\begin{split}
&(x^3 - x - 2) (x^3 + 2 x^2 + 3 x + 4)\\  
&\ \ + y^2\left[(x - 1) (3 x^3 + 7 x^2 + 7 x - 1)+y^2(3x^2 + 2x - 2+y^2)\right]> 0.
\end{split}
\end{align}
First we show that the expression in the square brackets of (\ref{P4sub}) is positive. Using the facts that $y^2\geq 4-x^2$ and that $2x^2+2x+2=2(x+\tfrac{1}{2})^2+\tfrac{3}{22}> 0$, we obtain
\begin{align*}
&\mathrel{\phantom{=}} (x- 1) (3 x^3 + 7 x^2 + 7 x - 1)+y^2(3x^2 + 2x - 2+y^2)\\
&\geq  (x - 1) (3 x^3 + 7 x^2 + 7 x - 1)+y^2(2x^2 + 2x+2)\\
&\geq  (x - 1) (3 x^3 + 7 x^2 + 7 x - 1)+(4-x^2)(2x^2+2x+2)\\
&= x^4 + 2x^3 + 6x^2 + 9.
\end{align*}
Now one can verify that $x^4+2x^3+6x^2+9> 0$ for all $x\in\mathbb{R}$. We use the bound $y^2\geq \max\{4-x^2,0\}$ to bound the left side of (\ref{P4sub}) as follows:
\begin{align*}
&\mathrel{\phantom{=}} (x^3 - x - 2) (x^3 + 2 x^2 + 3 x + 4) \\
&\mathrel{\phantom{=}} \ \ +y^2\left[(x - 1) (3 x^3 + 7 x^2 + 7 x - 1)+y^2(3x^2 + 2x - 2+y^2)\right]\\
&\geq  (x^3 - x - 2) (x^3 + 2 x^2 + 3 x + 4)  + y^2(x^4 + 2x^3 + 6x^2 + 9).\\
&\geq \max\{(x^3 - x - 2) (x^3 + 2 x^2 + 3 x + 4)  + (4-x^2)(x^4 + 2x^3 + 6x^2 + 9),\\
&  \hspace{2cm} (x^3 - x - 2) (x^3 + 2 x^2 + 3 x + 4)\}\\
&=\max\{2(4 x^3 + 4 x^2 - 5 x + 14),(x^3 - x - 2) (x^3 + 2 x^2 + 3 x + 4)\}.
\end{align*}
The proof is completed by verifying that $2(4 x^3 + 4 x^2 - 5 x + 14)> 0$ for all $x\geq -2$, while $(x^3 - x - 2) (x^3 + 2 x^2 + 3 x + 4)>0$ for all $x<-2$.

\medskip

\noindent
\textbf{Case III:} $T\cong P_4$ and $v$ is not a leaf.  We prove the stronger statement that if $|z|\geq 2$, then $\left|\Phi_{T,v}(z)\right|> \left|\Phi_{T-v}(z)\right|$.  So let $|z|\geq 2$. Note that $\Phi_{T,v}(z)=z^4+2z^3+2z^2+z$ and $\Phi_{T-v}(z)=z^2+3z$, so we need to show that 
\begin{align}\label{P4Central}
|z^3+2z^2+2z+1|^2> |z+3|^2.
\end{align} 
Write $z=x+iy$, where $x,y\in\mathbb{R}$.  By substituting and simplifying, we find that (\ref{P4Central}) is equivalent to 
\begin{align}\label{P4CentralSub}
\begin{split}
&(x^3 + 2x^2 + 3x + 4)(x^3 + 2x^2 + x - 2)\\
&\ \ +y^2\left[(3x^3 + 5x^2 + 3x - 1)(x + 1)+y^2(3x^2 + 4x+y^2)\right]>0.
\end{split}
\end{align}
First we show that the expression in the square brackets of (\ref{P4CentralSub}) is positive.  We use the supposition that $y^2\geq 4-x^2$, and the fact that $2x^2+4x+4=2(x+1)^2+2>0$, to obtain
\begin{align*}
&\mathrel{\phantom{=}} (3x^3 + 5x^2 + 3x - 1)(x + 1)+y^2(3x^2 + 4x+y^2)\\
&\geq (3x^3 + 5x^2 + 3x - 1)(x + 1)+y^2(2x^2 + 4x+4)\\
&\geq (3x^3 + 5x^2 + 3x - 1)(x + 1)+(4-x^2)(2x^2 + 4x+4)\\
&= x^4 + 4 x^3 + 12 x^2 + 18 x + 15,
\end{align*}
which we can verify is positive for all $x\in\mathbb{R}$.  Using this fact along with the bound $y^2\geq \max\{4-x^2,0\}$, we bound the left side of (\ref{P4CentralSub}) as follows:
\begin{align*}
&\mathrel{\phantom{=}}(x^3 + 2x^2 + 3x + 4)(x^3 + 2x^2 + x - 2)\\
&\mathrel{\phantom{=}}\ \   +y^2\left[(3x^3 + 5x^2 + 3x - 1)(x + 1)+y^2(3x^2 + 4x+y^2)\right]\\
&\geq \max\{(x^3 + 2x^2 + 3x + 4)(x^3 + 2x^2 + x - 2)+(4-x^2)(x^4 + 4 x^3 + 12 x^2 + 18 x + 15),\\
& \hspace{2cm} (x^3 + 2x^2 + 3x + 4)(x^3 + 2x^2 + x - 2)\}\\
&=\max\{2(4 x^3 + 20 x^2 + 35 x + 26),(x^3 + 2x^2 + 3x + 4)(x^3 + 2x^2 + x - 2)\}.
\end{align*}
The proof is completed by verifying that $2(4 x^3 + 20 x^2 + 35 x + 26)>0$ for $x\geq -2$, while $(x^3 + 2x^2 + 3x + 4)(x^3 + 2x^2 + x - 2)>0$ for $x < -2$.
\end{proof}

We are now ready to prove the main result of this section.

\begin{theorem}\label{GlobalBound}
The subtree roots of every tree $T$ lie in the disk 
\[
D=\left\{z\in\mathbb{C}\colon\ |z| \leq 1 + \sqrt[3]{3}\right\},
\]
and the only tree with a subtree root on the boundary of $D$ is $K_{1,3}$.
\end{theorem}

\begin{proof}
Let $T$ be a tree with vertex $v$.  If $T\cong K_{1,3}$, then $\Phi_T\left(-1-\sqrt[3]{3}\right)=0$ (giving the root on the boundary of $D$), and by inspection, all other subtree roots of $K_{1,3}$ lie in the interior of $D$.

Now suppose that $T$ is not isomorphic to $K_{1,3}$.  Recall from (\ref{PhiDecomposition}) that we have $\Phi_T(z)=\Phi_{T,v}(z)+\Phi_{T-v}(z)$.  By the reverse triangle inequality,
\[
\left|\Phi_{T}(z)\right| \geq \left|\Phi_{T,v}(z)\right|-\left|\Phi_{T-v}(z)\right|.
\]
We claim that if $|z|\geq 1+\sqrt[3]{3}$, then 
\[
\left|\Phi_{T,v}(z)\right|>\left|\Phi_{T-v}(z)\right|,
\]
from which it follows immediately that the subtree roots of $T$ lie in the interior of $D$.
To prove the claim, we proceed by induction on the order of $T$.  First of all, if $T$ has order at most $4$, then the claim holds by Lemma~\ref{BaseCase}.  Now suppose, for some $n\geq 5$, that for every tree $S$ of order strictly less than $n$ not isomorphic to $K_{1,3}$, and every vertex $u$ of $S$, if $|z|\geq 1+\sqrt[3]{3}$, then $\left|\Phi_{S,u}(z)\right|>\left|\Phi_{S-u}(z)\right|$. By Lemma~\ref{K13Lemma}, we also know, for every vertex $u$ of $K_{1,3}$, that if $|z|\geq 1+\sqrt[3]{3}$, then $\left|\Phi_{K_{1,3},u}(z)\right|\geq \left|\Phi_{K_{1,3}-u}(z)\right|$.  So altogether, for every tree $S$ of order strictly less than $n$ (including $K_{1,3}$), and every vertex $u$ of $S$, if $|z|\geq 1+\sqrt[3]{3}$, then $\left|\Phi_{S,u}(z)\right|\geq \left|\Phi_{S-u}(z)\right|$.  For convenience, we refer to this last statement as the induction hypothesis.

Suppose that $T$ has order $n$.  We wish to show that if $|z|\geq 1+\sqrt[3]{3}$, then $\left|\Phi_{T,v}(z)\right|>\left|\Phi_{T-v}(z)\right|$.  So let $|z|\geq 1+\sqrt[3]{3}$, and let $u$ be a neighbour of $v$.  The graph $T-uv$ is a forest with exactly two components.  Let $T_1$ be the component of $T-uv$ containing $v$, and let $T_2$ be the component of $T-uv$ containing $u$.  Then
\begin{align}\label{v}
\Phi_{T,v}(z)=\Phi_{T_1,v}(z)\cdot\left(1+\Phi_{T_2,u}(z)\right),
\end{align}
and 
\begin{align}\label{nov}
\Phi_{T-v}(z)=\Phi_{T_1-v}(z)+\Phi_{T_2}(z).
\end{align}
At this point, we consider three cases.

\medskip

\noindent
\textbf{Case I:} The tree $T_1$ has order $1$.

\noindent
From (\ref{v}) and (\ref{nov}), we have $\Phi_{T,v}(z)=z\cdot\left(1+\Phi_{T_2,u}(z)\right)$ and $\Phi_{T-v}(z)=\Phi_{T_2}(z)$, so it suffices to show that
\[
|z|\cdot \left|1+\Phi_{T_2,u}(z)\right|>\left|\Phi_{T_2}(z)\right|.
\]
Now 
\[
|z|\cdot \left|1+\Phi_{T_2,u}(z)\right|\geq |z|\cdot \left(\left|\Phi_{T_2,u}(z)\right|-1\right)=|z|\cdot \left|\Phi_{T_2,u}(z)\right|-|z|,
\] 
and
\[
\left|\Phi_{T_2}(z)\right|\leq \left|\Phi_{T_2,u}(z)\right|+\left|\Phi_{T_2-u}(z)\right|\leq 2\left|\Phi_{T_2,u}(z)\right|
\]
by the induction hypothesis.  Therefore, it suffices to show that
\[
|z|\cdot \left|\Phi_{T_2,u}(z)\right|-|z|> 2\left|\Phi_{T_2,u}(z)\right|,
\]
or equivalently,
\[
\left(|z|-2\right)\cdot \left|\Phi_{T_2,u}(z)\right|> |z|.  
\]
By Lemma~\ref{LocalBound}, since $T_2$ has order $n-1$, we have
\begin{align*}
\left(|z|-2\right)\cdot \left|\Phi_{T_2,u}(z)\right|&\geq \left(|z|-2\right)\cdot\left(|z|-1\right)^{n-2}\cdot |z|\\
&\geq \left(\sqrt[3]{3}-1\right)\cdot \left(\sqrt[3]{3}\right)^{3}\cdot |z|\\
&>|z|,
\end{align*}
where we used $n\geq 5$ and $|z|\geq 1+\sqrt[3]{3}$.

\medskip

\noindent
\textbf{Case II:} The tree $T_2$ has order $1$.

\noindent
From (\ref{v}) and (\ref{nov}), we have $\Phi_{T,v}(z)=\Phi_{T_1,v}(z)\cdot(1+z)$ and $\Phi_{T-v}(z)=\Phi_{T_1-v}(z)+z.$  Thus
\[
\left|\Phi_{T,v}(z)\right|\geq \left|\Phi_{T_1,v}(z)\right|\cdot \left(|z|-1\right)=\left|\Phi_{T_1,v}(z)\right|+\left(|z|-2\right)\cdot\left|\Phi_{T_1,v}(z)\right|,
\]
and
\[
\left|\Phi_{T-v}(z)\right|\leq \left|\Phi_{T_1-v}(z)\right|+|z|\leq  \left|\Phi_{T_1,v}(z)\right|+|z|,
\]
where the induction hypothesis was used at the last inequality.  Thus, it suffices to show that
\[
\left(|z|-2\right)\cdot\left|\Phi_{T_1,v}(z)\right|> |z|.
\]
This follows by an argument analogous to that used in Case I.

\medskip

\noindent
\textbf{Case III:} Both $T_1$ and $T_2$ have order at least $2$.

\noindent
From (\ref{v}), we have
\begin{align*}
\left|\Phi_{T,v}(z)\right|&\geq \left|\Phi_{T_1,v}(z)\right|\cdot\left(\left|\Phi_{T_2,u}(z)\right|-1\right)\\
&=\left|\Phi_{T_1,v}(z)\right|+ \left|\Phi_{T_1,v}(z)\right|\cdot\left(\left|\Phi_{T_2,u}(z)\right|-2\right).
\end{align*}
From (\ref{nov}), we have
\begin{align*}
\left|\Phi_{T-v}(z)\right|&\leq \left|\Phi_{T_1-v}(z)\right|+\left|\Phi_{T_2}(z)\right|\\
&\leq \left|\Phi_{T_1-v}(z)\right|+\left|\Phi_{T_2,u}(z)\right|+\left|\Phi_{T_2-u}(z)\right|\\
&\leq \left|\Phi_{T_1,v}(z)\right|+2\left|\Phi_{T_2,u}(z)\right|,
\end{align*}
where the induction hypothesis was used at the last inequality.  So it suffices to show that
\[
\left|\Phi_{T_1,v}(z)\right|\cdot\left(\left|\Phi_{T_2,u}(z)\right|-2\right)> 2\left|\Phi_{T_2,u}(z)\right|,
\]
or equivalently,
\[
\left(\left|\Phi_{T_1,v}(z)\right|-2\right)\left(\left|\Phi_{T_2,u}(z)\right|-2\right)> 4.
\]
Let $n_i$ denote the order of $T_i$ for $i\in\{1,2\}$.  Since $n\geq 5$ and $n_1,n_2\geq 2$, we must have either $n_1\geq 3$ and $n_2\geq 2$, or $n_1\geq 2$ and $n_2\geq 3$. Using this fact along with Lemma~\ref{LocalBound} and the inequality $|z|\geq 1+\sqrt[3]{3}$, we have
\begin{align*}
\left(\left|\Phi_{T_1,v}(z)\right|-2\right)\left(\left|\Phi_{T_2,u}(z)\right|-2\right)&\geq \left(|z|\cdot\left(|z|-1\right)^{n_1-1}-2\right)\left(|z|\cdot\left(|z|-1\right)^{n_2-1}-2\right)\\
&\geq \left(|z|\cdot\left(|z|-1\right)^2-2\right)\left(|z|\cdot\left(|z|-1\right)-2\right)\\
&\geq \left((1+\sqrt[3]{3})\left(\sqrt[3]{3}\right)^2-2\right)\left((1+\sqrt[3]{3})\sqrt[3]{3}-2\right)\\
&>4.
\end{align*}
This completes the proof of the theorem.
\end{proof}

\section{Real subtree roots}\label{RealLineSection}

Since the coefficients of the subtree polynomial are all positive, it follows immediately that $(0,\infty)$ is a root-free interval of the real line for the subtree polynomial.  By Theorem~\ref{GlobalBound}, the interval $\left(-\infty,-1-\sqrt[3]{3}\right)$ is also free of subtree roots.  Moreover, both intervals are maximal in this sense, as $-1-\sqrt[3]{3}$ and $0$ are both subtree roots of $K_{1,3}$ (in fact, every tree has subtree root $0$).  In this section, we prove that the interval $[-1,0)$ is also free of subtree roots.  It follows that the collection $\mathcal{R}$ of real subtree roots of all trees is contained in $[-1-\sqrt[3]{3},-1]\cup\{0\}$.  We demonstrate that the closure of $\mathcal{R}$ contains the interval  $[-2,-1]$.  Hence, the interval $[-1,0)$ is a maximal root-free interval for the subtree polynomial as well.

First, we show that the interval $[-1,0)$ is free of subtree roots.  We actually prove the stronger result that the subtree polynomial is increasing on the interval $(-1,0)$, from which the desired result easily follows.  The proof relies on the useful fact that $x\Phi'_{T}(x)$ is equal to the sum of all local subtree polynomials of $T$.  This fact was observed by Jamison~\cite{Jamison1987}, but we provide justification below for completeness.

\begin{theorem}\label{Derivative}
Let $T$ be a tree.  Then $\Phi'_T(x)> 0$ for $x\in(-1,0)$.
\end{theorem}
\begin{proof}
Let $T$ be a tree of order $n$ with vertex set $V$, and let $\Phi_T(x)=\displaystyle\sum_{k=1}^ns_kx^k$.  First, note that 
\begin{align*}
x\Phi'_T(x)=\sum_{k=1}^n ks_kx^k=\sum_{v\in V}\Phi_{T,v}(x).
\end{align*}
This follows from the fact that every subtree $S$ of $T$ of order $k$ is counted exactly $k$ times in the sum on the right (once for each of the $k$ vertices in $S$).  So to show that $\Phi'_T(x)>0$ for all $x\in(-1,0)$, it suffices to show that $\Phi_{T,v}(x)<0$ for all $v\in V$ and $x\in (-1,0)$.  We prove the stronger statement that $-1<\Phi_{T,v}(x)<0$ for all $v\in V$ and $x\in(-1,0)$.  The proof is by induction on the order of $T$.

The statement is easily verified for a tree of order $1$.  Suppose that the statement holds for all trees of order strictly less than $n$.  Let $T$ be a tree of order $n$ and let $v$ be an arbitrary vertex of $T$.  Let $v$ have neighbours $u_1,\dots,u_d$, where $d=\deg(v)$.  Recall from (\ref{PhiProduct}) that
\begin{align}\label{Recursion}
\Phi_{T,v}(x)&=x\cdot\prod_{i=1}^d \left[1+\Phi_{T_i,v_i}(x)\right],
\end{align}
where $T_i$ is the component of $T-v$ containing $v_i$.  Now suppose that $x\in(-1,0)$. Then $0<1+\Phi_{T_i,v_i}(x)<1$ for all $i\in\{1,\dots,d\}$ by the induction hypothesis.  From~(\ref{Recursion}), we conclude that $-1<\Phi_{T,v}(x)<0$.
\end{proof}

\begin{theorem}
No tree has a real subtree root in the interval $[-1,0)$.
\end{theorem}

\begin{proof}
Let $T$ be a tree.  We show that $\Phi_T(x)<0$ for all $x\in[-1,0)$.  Note that $\Phi_T(0)=0$, and by Theorem~\ref{Derivative}, $\Phi'_T(x)>0$ for all $x \in(-1,0)$.  Applying the Mean Value Theorem, we find $\Phi_T(x)<0$ for all $x\in[-1,0)$.
\end{proof}

Next, we demonstrate that the closure of the collection of real subtree roots of all trees contains the interval $[-2,-1]$, i.e.\ subtree roots are dense in $[-2,-1]$.  We will use the following result from elementary real analysis.

\begin{proposition}[{\cite[Proposition 6.4.5]{Garling2013}}]\label{ContinuousInverse}
If $f$ is a strictly monotonic function on an interval $I$, then  $f^{-1}:f(I)\rightarrow I$ is continuous.
\end{proposition}

\begin{theorem}\label{Closure}
The closure of the collection of real subtree roots of all trees contains the interval $[-2,-1]$.
\end{theorem}

\begin{proof}
Let $T_{a,b}$ denote the tree on $a+2b+1$ vertices obtained from $K_{1,a+b}$ by joining a pendant vertex to exactly $b$ leaves (see Figure~\ref{Tab}).  We prove that the closure of the collection of real subtree roots of all trees in the set $\{T_{a,b}\colon\ a,b\geq 1, a \mbox{ odd}\}$ contains $[-2,-1]$.  

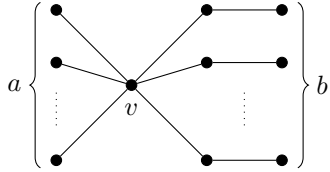
\begin{figure}[htb]
\centering{
\begin{tikzpicture}
\vertex (0) at (0,0) {};
\vertex (1) at (-1,1) {};
\vertex (2) at (-1,0.3) {};
\vertex (3) at (-1,-1) {};
\vertex (4) at (1,1) {};
\vertex (4a) at (2,1) {};
\vertex (5) at (1,0.3) {};
\vertex (5a) at (2,0.3) {};
\vertex (6) at (1,-1) {};
\vertex (6a) at (2,-1) {};
\path
(0) edge (1)
(0) edge (2)
(0) edge (3)
(0) edge (4)
(0) edge (5)
(0) edge (6)
(4) edge (4a)
(5) edge (5a)
(6) edge (6a);
\draw[dotted] (-1,-0.1)--(-1,-0.6);
\draw[dotted] (1.5,-0.1)--(1.5,-0.6);
\draw [decorate,decoration={brace,amplitude=4pt,mirror},xshift=-6pt,yshift=0pt]
(-1,1.1) -- (-1,-1.1) node [left,black,midway,xshift=-3pt]
{\footnotesize $a$};
\draw [decorate,decoration={brace,amplitude=4pt},xshift=6pt,yshift=0pt]
(2,1.1) -- (2,-1.1) node [right,black,midway,xshift=3pt]
{\footnotesize $b$};
\node[below] at (0,-0.1) {\footnotesize $v$};
\end{tikzpicture}
}
\caption{The tree $T_{a,b}$.}\label{Tab}
\end{figure}

Let $a,b\geq 1$ with $a$ odd, and let $v$ be the central vertex of $T_{a,b}$, as labeled in Figure~\ref{Tab}.  Then
\begin{align*}
\Phi_{T_{a,b}}(x)&=\Phi_{T_{a,b},v}(x)+ax+b(x^2+2x)\\
&=x(1+x)^a(1+x+x^2)^b+ax+b(x^2+2x)\\
&=x\left[(1+x)^a(1+x+x^2)^b+a+b(x+2)\right]
\end{align*}
For convenience, we substitute $x=-1-y$, and use the fact that $a$ is odd to simplify.
\begin{align}
\Phi_{T_{a,b}}(-1-y)=(y+1)\left[y^a(y^2+y+1)^b-a-b(1-y)\right]
\end{align}
Let $f_{a,b}(y)=y^a(y^2+y+1)^b-a-b(1-y)$.  It suffices to show that the closure of the collection of real roots of all polynomials in $\{f_{a,b}(y)\colon\ a,b\geq 1, a\mbox{ odd}\}$ contains $(0,1)$.

For a fixed $a,b\geq 1$ with $a$ odd, consider the sequence of functions 
\begin{align*}
f_{an,bn}(y)&=y^{an}(y^2+y+1)^{bn}-an-bn(1-y)\\
&=\left[y^a(y^2+y+1)^b\right]^n-[a+b(1-y)]n,
\end{align*}
where $n\geq 1$ and $n$ is odd (so that $an$ is odd).  Note that $f_{an,bn}$ is increasing for all $y> 0$.  Since $f_{an,bn}(0)=-(a+b)n<0$, and $\displaystyle\lim_{y\rightarrow\infty}f_{an,bn}(y)=\infty$, it follows that $f_{an,bn}(y)$ has a unique positive root.  Further, if $y^a(y^2+y+1)^b>1,$ then
\[
\lim_{n\rightarrow \infty}f_{an,bn}(y)=\infty,
\]
while if $y^a(y^2+y+1)<1$, then
\[
\lim_{n\rightarrow \infty}f_{an,bn}(y)=-\infty.
\]
Therefore, for $n$ sufficiently large, the unique positive root of $f_{an,bn}(y)$ can be made arbitrarily close to the unique positive root of the function
\[
g_{a,b}(y)=y^a(y^2+y+1)^b-1,
\]
which, evidently, lies in $(0,1)$.  Let $\mathcal{R}_g\subseteq (0,1)$ denote the collection of positive roots of all polynomials in $\{g_{a,b}(y)\colon\ a,b\geq 1, a\mbox{ odd}\}$.  It now suffices to show that $\mathcal{R}_g$ is dense in $(0,1)$.

Let $r=r_{a,b}$ denote the unique root of $g_{a,b}(y)$ in $(0,1)$, and let $t=\tfrac{a}{a+b}$.  Then
\[
r^a(r^2+r+1)^b=1 \Leftrightarrow r^{\tfrac{a}{a+b}}(r^2+r+1)^{\tfrac{b}{a+b}}=1 \Leftrightarrow r^t(r^2+r+1)^{1-t}=1.
\]
For every $t\in(0,1)$, define
\[
h_t(y)=y^t(y^2+y+1)^{1-t}-1.
\]
and let $R(t)$ denote the unique root of $h_t(y)$ in $(0,1)$ (so that $R\left(\frac{a}{a+b}\right) = r_{a,b}$).   We claim that $R$ is a continuous function of $t$ on $(0,1)$ with range $(0,1)$.  For a fixed $t$, let $\rho=R(t)$, so that $\rho^t(\rho^2+\rho+1)^{1-t}=1$. Taking logarithms on both sides and solving for $t$, we find 
\[
t=\frac{\log(\rho^2+\rho+1)}{\log\left(\rho+1+\tfrac{1}{\rho}\right)}.
\]
It is easily checked that this expression for $t$ is a strictly increasing function of $\rho$ with domain $(0,1)$ and range $(0,1)$.  In other words, $R^{-1}:(0,1)\rightarrow (0,1)$ is a strictly increasing function, and thus by Proposition~\ref{ContinuousInverse}, we conclude that $R:(0,1)\rightarrow (0,1)$ is a continuous function. Since the set 
\[ Q=\left\{ \tfrac{a}{a+b} : a \mbox{ and } b \mbox{ are positive integers, a odd} \right\}\]
in dense in $(0,1)$, and $\mathcal{R}_g$ is precisely the image of $Q$ under the continuous map $R$, we conclude that $\mathcal{R}_g$ is dense in the range of $R$, namely $(0,1)$.  This completes the proof.
\end{proof}

\section{Concluding remarks and open problems}\label{Conclusion}

In this article, we demonstrated that the roots of the subtree polynomial lie in the disk $D=\left\{z\in\mathbb{C}\colon\ |z|\leq 1+\sqrt[3]{3}\right\}$, and that $K_{1,3}$ is the unique tree with a subtree root on the boundary of $D$.  We make the following stronger conjecture, which we have verified for all $n\leq 18$.

\begin{conjecture}\label{DiskConjecture}
Let $T$ be a tree of order $n\geq 2$.  Then the subtree roots of $T$ lie in the disk $D_n=\left\{z\in\mathbb{C}\colon\ |z|\leq 1+\sqrt[n-1]{n-1}\right\}$.
\end{conjecture}

For $n\leq 18$, the star $K_{1,n-1}$ is the unique tree with a subtree root of maximum modulus among all trees of order $n$, and we conjecture that this is true for all $n$.  It is straightforward to show that the subtree roots of $K_{1,n-1}$ lie in the disk $D_n$ for all $n\geq 2$.

We also found that the intervals $\left(-\infty,-1-\sqrt[3]{3}\right)$, $[-1,0)$, and $(0,\infty)$ are free of subtree roots, and that the closure of the collection $\mathcal{R}$ of real subtree roots of all trees contains the interval $[-2,-1]$.  What can be said about the real subtree roots in the interval $(-1-\sqrt[3]{3},-2)$?  If Conjecture~\ref{DiskConjecture} is correct, then for any real number $t\in (-1-\sqrt[3]{3},-2)$, there are only finitely many trees with subtree roots of modulus at least $|t|$, and hence only finitely many real subtree roots in the interval $(-1-\sqrt[3]{3},t]$.  Therefore, if Conjecture~\ref{DiskConjecture} is correct, then the closure of $\mathcal{R}$ is exactly $\mathcal{R}\cup[-2,-1]$.

\begin{figure}[!h]
    \centering
    \begin{tikzpicture}
    \node[inner sep=0pt,left] (plot) at (0,0)
    {\includegraphics[scale=0.5]{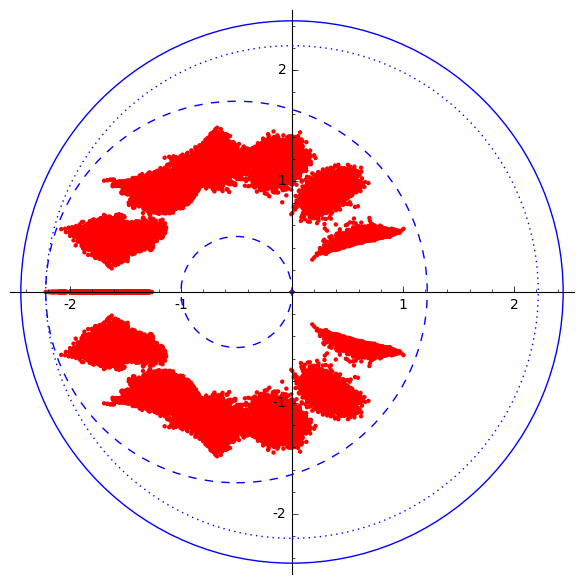}};
    \node[right] at (-0.05,-0.05) {\footnotesize $\mbox{Re}(z)$};
    \node[above] at (-3.7,3.55) {\footnotesize $\mbox{Im}(z)$};
    \end{tikzpicture}
    \caption{The subtree roots of all trees of order $14$, the disk $\left\{z\in\mathbb{C}\colon\ |z|=1+\sqrt[3]{3}\right\}$ of Theorem~\ref{GlobalBound} (solid line), the disk $\left\{z\in\mathbb{C}\colon\ |z|=1+\sqrt[13]{13}\right\}$ of Conjecture~\ref{DiskConjecture} (dotted line), and the annulus $\left\{z\in\mathbb{C}\colon\ \tfrac{1}{2}\leq \left|z+\tfrac{1}{2}\right|\leq \tfrac{1}{2}+\sqrt[13]{13}\right\}$ of Conjecture~\ref{AnnulusConjecture} (dashed lines).}
    \label{Roots14Conjecture}
\end{figure}

Finally, we note that the collection of complex subtree roots of all trees actually appears to be centred loosely around the point $-1/2$ in the complex plane.  In fact, we make the following conjecture that strengthens Conjecture~\ref{DiskConjecture}. 

\begin{conjecture}\label{AnnulusConjecture}
If $T$ is a tree of order $n\geq 2$, then the subtree roots of $T$ are contained in the annulus
\[
\left\{z\in\mathbb{C}\colon\ \tfrac{1}{2}\leq \left|z+\tfrac{1}{2}\right|\leq \tfrac{1}{2}+\sqrt[n-1]{n-1}\right\}.
\]
\end{conjecture}

We have confirmed Conjecture~\ref{AnnulusConjecture} for all $n\leq 18$.  The subtree roots of all trees of order $14$, along with the disk $D$ that we have proven contains all subtree roots, the disk of Conjecture~\ref{DiskConjecture}, and the annulus of Conjecture~\ref{AnnulusConjecture} are illustrated in Figure~\ref{Roots14Conjecture}.  What can be said about the closure of the collection of complex subtree roots of all trees?  Does it contain the entire annulus $\left\{z\in\mathbb{C}\colon\ \tfrac{1}{2}\leq \left|z+\tfrac{1}{2}\right|\leq \tfrac{3}{2}\right\}$?

\section*{Acknowledgements}

We wish to thank Shannon Ezzat for a conversation which was helpful in completing the proof of Theorem~\ref{Closure}.

\providecommand{\bysame}{\leavevmode\hbox to3em{\hrulefill}\thinspace}
\providecommand{\MR}{\relax\ifhmode\unskip\space\fi MR }
\providecommand{\MRhref}[2]{%
  \href{http://www.ams.org/mathscinet-getitem?mr=#1}{#2}
}
\providecommand{\href}[2]{#2}


\begin{thebibliography}{10}

\bibitem{BrownMol2016}
J.~I. Brown and L. Mol, \emph{{On the roots of the node reliability
  polynomial}}, Networks \textbf{68}({3}) (2016), 238--246.

\bibitem{BrownMol2017}
J.~I. Brown and L. Mol, \emph{{On the roots of all-terminal reliability
  polynomials}}, Discrete Math. \textbf{340}({6}) (2017), 1287--1299.

\bibitem{BrownTufts2014}
J.~I. Brown and J. Tufts, \emph{{On the Roots of Domination Polynomials}},
  Graphs Combin. \textbf{30}({3}) (2014), 527--547.

\bibitem{CsikvariOboudi2011}
P. Csikv{\'{a}}ri and M.~R. Oboudi, \emph{{On the roots of edge cover
  polynomials of graphs}}, European J. Combin. \textbf{32}({8}) (2011),
  1407--1416.

\bibitem{DongKohTeo2005}
F.~M. Dong, K.~M. Koh, and K.~L. Teo, \emph{{Chromatic Polynomials and
  Chromaticity of Graphs}}, World Scientific Publishing, Singapore, 2005.

\bibitem{Garling2013}
D.~J.~H. Garling, \emph{{A Course in Mathematical Analysis, Volume I:
  Foundations and Elementary Real Analysis}}, Cambridge University Press, 2013.

\bibitem{Haslegrave2014}
J. Haslegrave, \emph{{Extremal results on average subtree density of
  series-reduced trees}}, J. Combin. Theory Ser. B \textbf{107}({1}) (2014),
  26--41.

\bibitem{Jamison1983}
R.~E. Jamison, \emph{{On the average number of nodes in a subtree of a tree}},
  J. Combin. Theory Ser. B \textbf{35} (1983), 207--223.

\bibitem{Jamison1984}
R.~E. Jamison, \emph{{Monotonicity of the mean order of subtrees}}, J. Combin.
  Theory Ser. B \textbf{37} (1984), 70--78.

\bibitem{Jamison1987}
R.~E. Jamison, \emph{{Alternating Whitney sums and matchings in trees, part
  I}}, Discrete Math. \textbf{67}({2}) (1987), 177--189.

\bibitem{Jamison1990}
R.~E. Jamison, \emph{{Alternating Whitney sums and matchings in trees, part
  II}}, Discrete Math. \textbf{79}({2}) (1990), 177--189.

\bibitem{LevitMandrescu2005}
V.~E. Levit and E. Mandrescu, \emph{{The independence polynomial of a graph - a
  survey}}, Proceedings of the 1st International Conference on Algebraic
  Informatics, Aristotle University of Thessaloniki, Greece (2005), 231--252.

\bibitem{MolThesis}
L. Mol, \emph{{On connectedness and graph polynomials}}, Ph.D. thesis,
  Dalhousie University, 2016, pp.~1--174.

\bibitem{MolOellermann2017}
L. Mol and O. Oellermann, \emph{{Maximizing the mean subtree order}}, preprint,
  arXiv: 1707.01874 [math.CO] (2017), 1--31.

\bibitem{RalaivaosaonaWagner2018}
D. Ralaivaosaona and S. Wagner, \emph{{On the distribution of subtree orders of
  a tree}}, Ars Math. Contemp. \textbf{14}({1}) (2018), 129--156.

\bibitem{SillsWang2015}
A.~V. Sills and H. Wang, \emph{{The minimal number of subtrees of a tree}},
  Graphs Combin. \textbf{31}({1}) (2015), 255--264.

\bibitem{VinceWang2010}
A. Vince and H. Wang, \emph{{The average order of a subtree of a tree}}, J.
  Combin. Theory Ser. B \textbf{100}({2}) (2010), 161--170.

\bibitem{WagnerWang2014}
S. Wagner and H. Wang, \emph{{Indistinguishable trees and graphs}}, Graphs
  Combin. \textbf{30}({6}) (2014), 1593--1605.

\bibitem{WagnerWang2016}
S. Wagner and H. Wang, \emph{{On the local and global means of subtree
  orders}}, J. Graph Theory \textbf{81}({2}) (2016), 154--166.

\bibitem{YanYeh2006}
W. Yan and Y.-N. Yeh, \emph{{Enumeration of subtrees of trees}}, Theoret.
  Comput. Sci. \textbf{369}({1-3}) (2006), 256--268.

\bibitem{ZhangZhang2015}
X.-M. Zhang and X.-D. Zhang, \emph{{The minimal number of subtrees with a given
  degree sequence}}, Graphs Combin. \textbf{31}({1}) (2015), 309--318.

\bibitem{ZhangZhangGrayWang2013}
X.-M. Zhang, X.-D. Zhang, D. Gray, and H. Wang, \emph{{The number of subtrees
  of trees with given degree sequence}}, J. Graph Theory \textbf{73}({3})
  (2013), 280--295.

\end{thebibliography}
\end{document}